\documentclass[a4paper,9pt]{amsart}

\usepackage{mathrsfs,amsmath} 

\usepackage{amssymb}
\usepackage{cases}
  \usepackage{paralist}
  \usepackage{graphics}
  \usepackage{epsfig} 
\usepackage{graphicx}  \usepackage{epstopdf}
 \usepackage[colorlinks=true]{hyperref}
\hypersetup{urlcolor=blue, citecolor=red}

\newtheorem{theorem}{Theorem}
\newtheorem*{theorem*}{Theorem}
\newtheorem{corollary}{Corollary}

\newtheorem{lemma}{Lemma}
\newtheorem*{lemma*}{Lemma}

\theoremstyle{definition}

\newtheorem{remark}{Remark}
\newtheorem{notation}{Notation}
\newcommand{\ep}{\varepsilon}

\newcommand{\lin}{{\mathrm{lin}}}
\newcommand{\Rn}{{\mathbb{R}^n}}

\newcommand{\xii}{|\xi|}
\newcommand{\R}{{\mathbb{R}}}
\newcommand{\N}{{\mathbb{N}}}
\newcommand{\supp}{{\mathrm{\,supp\,}}}
\DeclareMathOperator{\sinc}{sinc}
\renewcommand{\doteq}{{\,\mathrm{:=}\,}}

\title[On the two critical exponents for a dispersive equation with weak dissipation] {On the two critical exponents for a dispersive equation with weak dissipation}
      
\author[M. D'Abbicco, A. Lagioia]{Marcello D'Abbicco, Antonio Lagioia}

\address{Marcello D'Abbicco, Department of Mathematics, University of Bari, Via E. Orabona 4, 70125 Bari, Italy}

\address{Antonio Lagioia, International School for Advanced Studies (SISSA), Via Bonomea 265, 34136 Trieste, Italy}

\email{marcello.dabbicco@uniba.it, alagioia@sissa.it}


\begin{document}
\begin{abstract}
In this paper, we study a nonlinear Cauchy problem written in normal form
\begin{equation*}
\begin{cases}
u_{tt} + A u+ Bu_t =D|u|^\alpha, & t\geq 0, x \in \Rn\\
u(0,x)=u_0(x),\,u_t(0,x)=u_1(x).
\end{cases}
\end{equation*}
Here~ $Af=\mathscr{F}^{-1}(a(\xi)^2\mathscr{F}f)$, with $a(\xi)$ homogeneous of order $k > 0$ under the assumption that the Hessian matrix of $a(\xi)$ is nonsingular,  the operator $Bf=\mathscr{F}^{-1}(b(\xi)\mathscr{F}f)$ represents a weak dissipation, in the sense that $b(\xi) \sim 1$ as $|\xi| \to \infty$ and
$Df=\nu\ast (-\Delta)^\delta(I-\Delta)^{-\delta}$, with $\nu$ a bounded measure and $\delta \geq k/2$. We derive $L^p-L^q$ estimates of the linear solution operator to the considered model and we prove the transition from a Fujita-type exponent to a Strauss-type exponent for the global-in-time existence of small data solution for the corresponding nonlinear problems. 
\end{abstract}


\maketitle
\section{Introduction}

In this paper, we are interested in studying the following nonlinear Cauchy problem
\begin{equation}\label{eq:CP:nonlinear}
\begin{cases}
u_{tt} + A u+ Bu_t  =Df(u), & t\geq 0, x \in \Rn\\
u(0,x)=u_0(x),\,u_t(0,x)=u_1(x),
\end{cases}
\end{equation}
where
\begin{equation}\label{eq:fu}
f(0)=0, \qquad |f(v)-f(u)|\leq C|v-u|(|v|^{\alpha-1}+|u|^{\alpha-1}),
\end{equation}
for some $\alpha>1$ (for instance, $f(u)=|u|^\alpha$).
We consider an operator $A$ defined by
\begin{equation}\label{eq:A}
Af = \mathscr{F}^{-1}(a^2(\xi)\,\hat f),\qquad f\in \mathcal S,
\end{equation}
and extended by density to other functional spaces, where we assume that $a\in\mathcal C^{n+1}(\R^n\setminus\{0\})\cap \mathcal C^{n+1,1}(S^{n-1})$, is homogeneous of degree~$k>0$ and it is an elliptic symbol, in the sense that $a(\xi)>0$ for any $\xi\neq0$.
By $\mathcal C^{n+1}(\R^n\setminus\{0\})\cap C^{n+1,1}(S^{n-1})$ for a homogeneous function, we mean that $a\in\mathcal C^{n+1}(\R^n\setminus\{0\})$ and that its derivatives of order $n+1$ are Lipschitz-continuous on $S^{n-1}$. This regularity will appear in the Littman-type Lemma that we use to derive $L^1-L^\infty$ estimates (see later, Lemma~\ref{lem:Littman}). On the other hand, we assume 
\begin{equation}\label{eq:B}
D=\nu\ast (-\Delta)^\delta(I-\Delta)^{-\delta}, \qquad Bf=\mathscr{F}^{-1} (b\hat f), \qquad f\in \mathcal S,
\end{equation}
where $\delta\geq k/2$, $\nu$ is a bounded measure and $b\in\mathcal C^{n+1}(\R^n\setminus\{0\})$ verifies: 
%
\begin{equation}\label{eq:hyp:b}
\begin{split}
& 0< b(\xi)\leq \min\{ C,  a(\xi) \},\qquad \liminf _{|\xi| \to 0} |\xi|^{-2\delta} b(\xi)>0, \qquad \liminf _{|\xi| \to +\infty} b(\xi)>0,\\
& |\partial_\xi^\alpha b(\xi)|\leq \xii^{-|\alpha|}\frac{\xii^{2\delta}}{1+\xii^{2\delta+\varepsilon}},\qquad |\alpha|=1,\ldots,n+1,
\end{split} 
\end{equation}
%
for some $\varepsilon>0$.
Here $(-\Delta)^\delta$ is the fractional Laplace operator and $(I-\Delta)^{-\delta}$ is the Bessel potential (see \textsection \ref{sec:not}).
Examples of interest of the operator $b$ include
\begin{align} 
\label{eq:assumption_b_3}
     &b(\xi) =\,\frac{|\xi|^{2\delta}}{(1+|\xi|^2)^{\delta}},\\
  \label{eq:assumption_b_1}  & b(\xi) =\,\frac{|\xi|^{2\delta}}{1+|\xi|^{2\delta}},\\
    \label{eq:assumption b 2}  & b(\xi)=\frac{|\xi|^{2\delta}+ |\xi|^{2\lambda}}{1+|\xi|^{2\lambda}},\\
  \label{eq:assumption b 4}    &   b(\xi)=\frac{|\xi|^{2\delta}+ |\xi|^{2\lambda}}{(1+|\xi|^2)^{\lambda}},
\end{align}
for~$k/2\leq \delta < \lambda$.
In case \eqref{eq:assumption_b_3} we have $\varepsilon=2$;  in case \eqref{eq:assumption_b_1} we have $\varepsilon=2\delta$; in case \eqref{eq:assumption b 2}  we have $\varepsilon=2(\lambda-\delta)$ and in case \eqref{eq:assumption b 4} we have $\varepsilon=\min\{2,2(\lambda-\delta)\}$, where $\varepsilon$ is as in \eqref{eq:hyp:b}. 

Examples of interest of $D$ include the case $\nu=\delta_0$, the Dirac measure, that is, $D=B$, where $B$ is as in \eqref{eq:assumption_b_3}, and
\begin{equation}\label{eq: D ex2}
D= \big( c_1(-\Delta)^{\lambda_1}+c_2(-\Delta)^{\lambda_2}\big) (I-\Delta)^{-\lambda_2}, \qquad \lambda_2>\lambda_1\geq\delta. 
\end{equation}
Indeed, in this case
\[ \hat\nu = \frac{c_1|\xi|^{2(\lambda_1-\delta)}+ c_2|\xi|^{2(\lambda_2-\delta)}}{(1+|\xi|^2)^{\lambda_2-\delta}}, \]
On the one hand,
\[ \frac{|\xi|^{2(\lambda_2-\delta)}}{(1+|\xi|^2)^{\lambda_2-\delta}}, \]
is the Fourier transform of a bounded measure (see Lemma 2 in V, \textsection 3.2 in~\cite{Stein70}). Then, it remains to show that
\[ \frac{|\xi|^{2(\lambda_1-\delta)}}{(1+|\xi|^2)^{\lambda_2-\delta}}, \]
is also the Fourier transform of a measure. We distinguish two cases. If~$\lambda_1>\delta$, then it is sufficient to apply Lemma~\ref{lem:HF} with $\varepsilon=2\min\{\lambda_1-\delta,\lambda_2-\lambda_1\}$. If $\lambda_1=\delta$, then it is the Fourier transform of the Bessel kernel.
Previous examples still hold if we replace all the Laplace operators $-\Delta$ by an operator
\begin{equation}\label{eq:G}
Gf = \mathscr{F}^{-1}(g^2(\xi)\,\hat f),\qquad f\in \mathcal S,
\end{equation}
where $g$ is homogeneous and $g(\xi)>0$ for any $\xi\neq0$. However, the application of different anisotropic operators is more complicated and yields counterexamples; indeed, if we consider the case $\delta=1$ and
\[\widehat{D}=\frac{|\xi|^2+\xi_2^2}{1+|\xi|^2}=\frac{|\xi|^2}{1+|\xi|^2}\,\frac{\xi_1^2+2\xi_2^2}{\xii^2},\]
we have that
\[ \frac{\xi_1^2+2\xi_2^2}{\xii^2} =  1 + \frac{\xi_2^2}{\xii^2} = \hat \delta_0 - \widehat{R_2^2\delta_0}, \]
where $R_2$ is the Riesz transform, is not the Fourier transform of a measure.\\

The basic example for $A$ is a differential homogeneous operator of order $2k$, $k\in\N$, with constant coefficients,
\begin{equation}\label{eq:elliptic}
A=\sum_{|\alpha|=2k} a_\alpha \partial_x^\alpha,\qquad \text{with}\quad a(\xi)^2=\sum_{|\alpha|=2k} a_\alpha \xi^\alpha >0, \quad \forall \xi\neq0.
\end{equation}
Then $a\in\mathcal C^\infty(\R^n\setminus\{0\})$ is homogeneous of degree $k$. The fractional Laplace operator $(-\Delta)^k$ also fits our assumptions on $a(\xi)$. 
We also assume that $\det H_a(\xi)\neq0$ for any $\xi\in S^{n-1}$, where $H_a$ denotes the Hessian matrix of~$a$ . As a consequence of this assumption, we shall exclude the case $k=1$ (see later, Remark \ref{rmk:k=1}). The basic example is $A=(-\Delta)^k$ once again, since $\det H_{\xii^k}\neq0$ if $k\neq1$, for $\xi\neq0$.
When $B=0$, model \eqref{eq:CP:nonlinear} is referred to as $k$-evolution equation. 

In the case of differential operators, the name $k$-evolution equation goes back to \cite{Miz61}. %
Results of well-posedness in Sobolev spaces which depend on lower order terms are obtained by several authors, see~\cite{ABZ2012,ABZ2016,CC2010,CHR2008}; the Gevrey well-posedness has been investigated in~\cite{AAC2024}. 

The dissipation considered in this paper does not fit in the class of lower order terms generally considered for the well-posedness of the problem; also, we are interested in the use of Fourier analysis to obtain decay estimates for the problem~\eqref{eq:CP:nonlinear}.
The hypothesis $\delta \geq k/2$ narrows the study to the damped oscillations regime, that this is the case in which the dissipation does not cancel the oscillations at both low and high frequencies.

Model \eqref{eq:CP:nonlinear} includes the double dispersion model (see \cite{DADE,Polat, Wang})
\begin{equation}\label{eq:DDDE}
\begin{cases}
u_{tt} + A u_{tt}+ Au + A^2u + \mu Au_t = Af(u), & t\geq 0, x \in \Rn,\\
u(0,x)=u_0(x),\,u_t(0,x)=u_1(x),
\end{cases}
\end{equation}
where $\mu \in \{0,1\}$. Indeed, let $D$ be as \eqref{eq:B} with $\delta=k$ and $\nu= A (I -A)^{-1} (-\Delta)^{-k} (I-\Delta)^k$ and $B=0$ if $\mu=0$ or $B$ satisfying \eqref{eq:assumption_b_1}, \eqref{eq:G} with $g=a$ if $\mu=1$. Model \eqref{eq:DDDE} was recently studied in \cite{DAL25b}, obtaining the existence of a global-in-time solution for $\alpha > \gamma(r)$, where $\gamma(r)$ is defined as the positive solution to the quadratic equation
\begin{equation}\label{eq:strauss-r}
     \frac1\gamma =\frac{r}2\,\frac{\gamma-1}{\gamma+1},
\end{equation}
if $\mu=0$, and to
\begin{equation}\label{eq:strauss DDDE}
 \frac 1{\gamma} = \frac{1}2 \left( \frac{n}{k}+\frac r2 \right) \, \frac{{ \gamma}-1}{{ \gamma}+1}+\frac12,
\end{equation}
if $\mu=1$, where $r$ is the rank of Hessian matrix of the symbol $a(\xi)$. The critical exponent $\gamma(n)$ as in \eqref{eq:strauss-r} is called Strauss exponent, and was conjectured to be the critical exponent for Klein-Gordon equation, Schr\"odinger equation in~\cite{CDW08, NakanishiOzawa=2002} and was modified into $\gamma(n-1)$ for the wave equation in~\cite{Strauss}. This was later proved in several papers (see~\cite{GLS} and the references therein).
While in the undamped double dispersion equation (\eqref{eq:DDDE} with $\mu=0$) the Strauss exponent appears naturally (see later, Remark \ref{rmk:trans}) due to the purely dispersive nature of the fundamental solution, in the damped model there is a competition between the dispersive and dissipative effects. When dissipation is stronger than dispersion, the approach that uses dispersive decay estimates ($L^{1+\frac 1\alpha}-L^{1+\alpha}$ as in \cite{DAL25b}) is replaced by the use of estimates of the type $L^1-L^\alpha$). The use of $L^1$ regularity yields a Fujita-type critical exponent, as is the case of the damped wave equation (see \cite{TY01}).

In particular, we are interested in determining the conditions on the space dimension $n$ and on the parameters $k$ and $\delta$ under which the exponent of global existence moves from a Fujita-type to a Strauss-type. The condition for the switch of the exponent is (see later, Remark \ref{rmk:trans})
\begin{equation}\label{eq:trans_cond}
    n\left(\frac M2+\delta \right)+M\leq  0, \qquad M=-k+4\delta^2-4\delta k +k^2.
\end{equation}
This is the consequence of the following two theorems. Let $\tilde{\gamma}$ be the quantity defined as
\begin{equation}\label{eq:strauss exp mod}
    \frac n2\left(\frac 1{\delta} +1-\frac k{2\delta}\right)\left(\frac {\tilde {\gamma}-1}{\tilde {\gamma}+1}\right)+ \left(1-\frac k{2\delta} \right)= \frac 1{\tilde {\gamma}},
\end{equation}
and
\begin{equation}\label{eq:fuj exp}
   p_F=1+\frac{\frac n2 \left(1-\frac k{2\delta}\right)+\frac k{2\delta}}{\frac n{2\delta}+1-\frac k{2\delta}+\frac n2 \left(1-\frac k{2\delta}\right)} .
\end{equation}
 It is useful to introduce the following quantity for any $1 \leq p \leq q \leq \infty$
\begin{equation}\label{eq: d(p,q)}
d(p,q)=\frac nk\left(\frac1p-\frac1q\right)+n\,\max\left\lbrace \left(\frac12-\frac1p\right),\left(\frac1q-\frac12\right)\right\rbrace.
\end{equation}
For simplicity, we will consider the case $u_0=0$ in \eqref{eq:CP:nonlinear} in the following.

\begin{theorem}\label{thm:glob_ex}
Let $n \geq 1$ and assume that \eqref{eq:trans_cond} does not hold. Assume $1-2\delta/k\leq d(1+1/\alpha,1+\alpha) <1$ and
\begin{equation}\label{eq: strauss exp}
\tilde\gamma < \alpha < \min \left\lbrace \frac{\left(\frac 1k+\frac \delta k-\frac 12\right)n+1}{\left(\left(\frac 1k+\frac \delta k-\frac 12\right)n-1\right)_+}, \frac{\frac n2 +1}{\left(\frac n2-1\right)_+} \right\rbrace. 
\end{equation}
Then there is a constant $\varepsilon_0>0$ such that for any
\[u_1\in \mathcal{A} := L^{1+\frac 1 \alpha}\cap\dot H^{-2\delta, 1+\frac 1 \alpha},\]
with $ \|u_1\|_{\mathcal{A}}<\varepsilon_0$,
there exists a unique solution 
\begin{align*}
u \in L^\infty([0,\infty), L^{\alpha+1}), \end{align*}
to \eqref{eq:CP:nonlinear}, satisfying the following estimate
\begin{align}
&\|u(t,\cdot)\|_{L^{\alpha +1}} \leq C\varepsilon_0 (1+t)^{-\delta(\alpha)},
\end{align}
where
\begin{equation}\label{eq:delta_alpha:strauss}
  \delta(\alpha)<  \frac n2\left(\frac 1{\delta} +1-\frac k{2\delta}\right)\frac {\alpha-1}{\alpha+1}+ \left(1-\frac k{2\delta} \right),
\end{equation}
and the constant $C>0$ does not depend on the initial data.
\end{theorem}
Notice that condition $d(1+1/\alpha,1+\alpha)<1$ is satisfied due to \eqref{eq: strauss exp}.
\begin{theorem}\label{thm:glob_ex_fuj}
Let $n \geq 1$. Assume \eqref{eq:trans_cond}, $1-(2\delta/k)\leq d(1,\alpha) <1$ and
\begin{equation}\label{eq: fuj exp}
p_F < \alpha < \min\left\lbrace1+\frac{\frac n2 \left(1-\frac k{2\delta}\right)+\frac k{2\delta}}{\left(\frac n{2\delta}-\frac k{2\delta}+\frac n2 \left(1-\frac k{2\delta}\right)\right)_+}, \frac n{(n-k)_+}, \frac {2n}{(n-2)_+} \right\rbrace .
\end{equation}
Then there is a constant $\varepsilon_0>0$ such that for any
\[u_1\in \mathcal{A} := L^{1}\cap\dot H^{-2\delta, 1},\]
with $ \|u_1\|_{\mathcal{A}}<\varepsilon_0$,
there exists a unique solution
\begin{align*}
u \in L^\infty([0,\infty),L^{\alpha}), \end{align*}
to \eqref{eq:CP:nonlinear}, satisfying the following estimate
\begin{align*}
&\|u(t,\cdot)\|_{L^{\alpha}} \leq C\varepsilon_0 (1+t)^{-\delta(\alpha)}, 
\end{align*}
where
\begin{equation}\label{eq:delta_alpha:fuj}
  \delta(\alpha)<\frac n{2\delta}  \left(1-\frac 1{\alpha} \right)-\frac n{\alpha}\left(1-\frac k{2\delta} \right) +\left( \frac n2+1 \right)\left(1-\frac k{2\delta} \right),
\end{equation}
and the constant $C>0$ does not depend on the initial data.
\end{theorem}
We also mention that a similar transition result is expected for other damped evolution equations with $A|u|^\alpha$ as a nonlinear term; see, for instance, \cite{DAEL26} for some damped Boussinesq equation.

The critical exponents are determined based on the decay estimates of the solution to the linear problem, namely
\begin{equation}\label{eq:CP:linear}
\begin{cases}
u_{tt} + A u+ Bu_t  =0, & t\geq 0, x \in \Rn,\\
u(0,x)=0,\,u_t(0,x)=u_1(x).
\end{cases}
\end{equation}
Estimates for solutions to the damped $k$-evolution equation of the form \eqref{eq:CP:linear} have recently been obtained in \cite{DAE25}, for various types of damping terms $B$. However, the case where the operator $B$ satisfies only \eqref{eq:hyp:b} is not included among the models studied in \cite{DAE25}, due to its high-frequency behavior. 

In the following, we use cut-off functions $\chi\in\mathcal C_c^\infty(\R^n)$, with $\chi=1$ in a neighborhood of the origin.
The main tool to prove full range $L^p-L^q$ for the solution to \eqref{eq:CP:linear}, is the following lemma. 

\begin{lemma}\label{lem:high0}
Let $s\geq0$, $t>0$ and~$q\in[1,\infty]$. Assume that $b\in\mathcal C^{n+1}(\R^n\setminus\{0\})$ satisfies hypothesis \eqref{eq:hyp:b}.
Then
\[ \chi\,\xii^{s\delta}\,e^{-b(\xi)t}\in M_q^q, \qquad (1-\chi)\,\,e^{-b(\xi)t}\in M_q^q,\]
and we may estimate
\begin{align}\label{eq:hestlow}
\|\chi\,\xii^{s\delta}\,e^{-b(\xi)t}\|_{M_q^q}
    & \leq C\,(1+t)^{-\frac{s}2},\\
\label{eq:hesthigh}
\|(1-\chi)\,\,e^{-b(\xi)t}\|_{M_q^q}
    & \leq C\,e^{-ct},
\end{align}
for some $C,c>0$, independent of $t$.
\end{lemma}
In particular, for $s=0$, Lemma~\ref{lem:high0} gives
\begin{equation}\label{eq:hest}
\|e^{-b(\xi)t}\|_{M_q^q}\leq C.
\end{equation}

It is clear that 
\[ b(\xi) = \frac{a(\xi)^2}{2(1+a(\xi)^2)} = \frac12 -\frac1{2(1+a(\xi)^2)}\,. \]
verifies the previous assumptions with $\varepsilon=2k$, as it happens in \eqref{eq:DDDE} with $\mu=1$.
Finally, we state the theorem related to the decay estimates for the solution to \eqref{eq:CP:linear}.

\begin{theorem}\label{thm:hyp}
Let $1\leq p\leq q\leq \infty$, $d(p,q)$ is as in~\eqref{eq: d(p,q)}, and $B$ as in~\eqref{eq:B}. Assume
\begin{equation}\label{eq:reg_fr}
s=s(n,k,p,q)>d(p,q).
\end{equation}
Fix $\ell\geq0$ and assume that $u_1\in H^{(s+\ell/k-1)_+,p}$, and let $u$ be the solution to \eqref{eq:CP:linear}. Then
\begin{equation}\label{eq:pert_fr}
\| (-\Delta)^{\frac{\ell}2} u(t,\cdot)\|_{L^q} \lesssim\,t^{1-\frac{n}k\left(\frac1p-\frac1q\right)- \frac\ell{k} +\left(s+ \frac{\ell}k-1\right)_+ } \,\langle t \rangle^{-\left(\frac \ell{2\delta}+ \frac k{2\delta}(s-1) \right)_+}\,\|u_1\|_{H^{k\left(s+\frac{\ell}k-1\right)_+,p}},
\end{equation}
for any $t>0$, with $C>0$ independent of $t$. Here 
$(\kappa)_+=\max\{\kappa,0\}$ for any $\kappa \in \R$.
\end{theorem}
The following immediate consequence of Theorem~\ref{thm:hyp} is of interest for the application to treat the nonlinear perturbation.
\begin{corollary}\label{cor:hyp}
Let $1\leq p\leq q\leq \infty$, $d(p,q)$ is as in~\eqref{eq: d(p,q)}, and assume~\eqref{eq:reg_fr}. Let $B$ and $D$ as in~\eqref{eq:B}. Assume that $u_1\in H^{(s-1)_+,p}$, and let $u$ be the solution to \eqref{eq:CP:linear}. Then
\begin{equation}\label{eq:estimate}
\| D u(t,\cdot)\|_{L^q} \lesssim\,t^{1-\frac{n}k\left(\frac1p-\frac1q\right)- \frac{2\delta}{k} +\left(s+ \frac{2\delta}k-1\right)_+ } \,\langle t \rangle^{-\left(1+ \frac k{2\delta}(s-1) \right)_+}\,\|u_1\|_{H^{k\left(s-1\right)_+,p}},
\end{equation}
for any $t>0$, with $C>0$ independent of $t$.
\end{corollary}
\begin{proof}
We first apply Theorem~\ref{thm:hyp} with $\ell=2\delta$. The thesis follows, due to the fact that $J_{2\delta}$ maps $H^{k\left(s+\frac{2\delta}k-1\right)_+,p}$ into $H^{k\left(s-1\right)_+,p}$ and that $\nu\ast$ is a bounded operator on $L^p$ for any $p\in[1,\infty]$ (see~\cite{Ho}).

\end{proof}

\begin{remark}\label{rmk:trans}
    Theorem \ref{thm:hyp} explains the reason behind the condition \eqref{eq:trans_cond}. We consider the estimate \eqref{eq:pert_fr} with $\ell=2\delta$, $q=p\alpha\leq p'$ and $s=d(p,q)+\varepsilon$, for some small enough $\varepsilon>0$. Moreover, we assume
\begin{equation}\label{eq:ass1}
    s \geq 1-\frac{2\delta}k.
\end{equation}
We obtain from \eqref{eq:pert_fr} that the decay rate for $t\geq 1$ is
\begin{align*}
&-(1-s)\left(1-\frac k{2\delta} \right)-\frac nk  \left(\frac 1p-\frac 1{p\alpha} \right) 
\\ &\qquad = -\frac n{2\delta}  \left(\frac 1p-\frac 1{p\alpha} \right) +n\left(1-\frac k{2\delta} \right)  \left(\frac 1{p\alpha}-\frac 12 \right) -\left(1-\frac k{2\delta} \right) +\ep \left(1-\frac k{2\delta} \right) \\
& \qquad=\frac{1}{p}\left( -\frac n{2\delta}  \left(1-\frac 1{\alpha} \right)+\frac n{\alpha}\left(1-\frac k{2\delta} \right) \right) -\left( \frac n2+1 \right)\left(1-\frac k{2\delta} \right)  +\ep \left(1-\frac k{2\delta} \right) 
\end{align*}
We want to determine the exponent $p \in [1,1+1/\alpha]$ that maximizes
%
%
\[\frac{1}{p}\left( \frac n{2\delta}  \left(\alpha-1\right)-n\left(1-\frac k{2\delta} \right) \right) +\alpha \left( \frac n2+1 \right)\left(1-\frac k{2\delta} \right)  -\alpha\ep \left(1-\frac k{2\delta} \right).\]
We have to use $p=1$ if $\alpha-1\geq 2\delta-k$. Then, the inequality $p_F\geq 1+2\delta-k$, where $p_F$ is as \eqref{eq:fuj exp}, is satisfied if and only if condition \eqref{eq:trans_cond} holds.
\end{remark}

So the global existence exponent is
 \[\min\{\tilde \gamma, p_F\},\]
 that is we have a continuous transition from the Fujita exponent, when  $n\left(\frac M2+\delta \right)+M\leq  0$ to the Strauss one, when $ n\left(\frac M2+\delta \right)+M>  0$.
 By the term transition, we mean that, if we fix $n$ and $k$ (or $k$ and $\delta$), we move from the Fujita exponent for $\delta$ small enough ($n$ respectively) to the Strauss type exponent for large $\delta$ ($n$ respectively).

On the other hand, considering the undamped problem, as showed in \cite{DAL25b}, the decay rate of the solution is
\[ -\frac{n}k\left(\frac1p-\frac1q\right)+d(p,q)=n\,\max\left\lbrace \frac12-\frac1p,\frac1q-\frac12\right\rbrace, \]
so that $L^{1+\frac1\alpha}-L^{1+\alpha}$ estimate is always the best one.

\begin{remark} \label{rmk:k=1}
    In the case $k=1$, the Hessian $H_a$ is necessarily degenerate, since $a$ is homogeneous; the maximal rank assumption for it means that the rank of $H_a$ is $n-1$. In this assumption, following as in~\cite{SteinHarmonic}, p. 428, it is likely possible to extend the result with $B=0$ in \eqref{eq:CP:linear}, upon replacing $d(p,q)$ in~\eqref{eq: d(p,q)} by
\begin{equation}\label{eq: d'(p,q)}
d'(p,q)=n\left(\frac1p-\frac1q\right)+(n-1)\,\max\left\lbrace \left(\frac12-\frac1p\right),\left(\frac1q-\frac12\right)\right\rbrace.
\end{equation}
The extension of Theorem~\ref{thm:hyp}, however, does not look equally straightforward, so we postpone this study to future investigations.
\end{remark}

The plan of the paper is the following:

In~\textsection\ref{sec:not}, we introduce the notations that we will use in the paper.

In~\textsection\ref{sec:lin est}, we prove the linear decay estimates stated in Theorem \ref{thm:hyp} and Lemma \ref{lem:high0}.

In~\textsection\ref{sec:nonlinear}, we prove Theorems \ref{thm:glob_ex} and \ref{thm:glob_ex_fuj}.

In~\textsection\ref{sec:App}, we collect results of the Fourier analysis and provide a proof of Littman lemma with less stringent requirements on the regularity of the phase and amplitude functions.

\subsection{Notation}\label{sec:not}

In this paper, we use the following notation.
\begin{notation}
For any multi-index $\alpha\in\N^n$, we use the notation
\[ \partial_x^\alpha= \partial_{x_1}^{\alpha_1}\ldots\partial_{x_n}^{\alpha_n}, \qquad \xi^\alpha = \xi_1^{\alpha_1}\ldots \xi_n^{\alpha_n}, \qquad |\alpha|=\alpha_1+\ldots+\alpha_n. \]
\end{notation}

\begin{notation}\label{not:F}
Let~$\widehat{f} = \mathscr{F}f$ be the Fourier transform, with respect to the~$x$ variable, of a tempered distribution $f$.  When $f\in L^1(\R^n)$, then
\[ \hat f(\xi)=\int_{\R^n} e^{-ix\xi}\,f(x)\,dx, \quad \text{where $x\xi$ denotes the inner product $\sum_j x_j\xi_j$ in $\R^n$.} \]
If $f\in L^1(\R^n)$, then~$\mathscr{F}^{-1}f(x)=(2\pi)^{-n}\mathscr{F}f(-x)$.
\end{notation}
\begin{notation}
Fractional powers of $-\Delta$ may be defined by their symbol in Fourier transform $(-\Delta)^s f=\mathscr{F}^{-1}(\xii^{2s}\mathscr{F}f)$ for $f \in H^s$ and~$s\in (0,\infty)$.
\end{notation}

\begin{notation}\label{not:Hsp}
For any~$p \in [1,\infty]$, and for~$s\in [0,\infty)$, $H^{s,p}$ denotes the Bessel potential space
\begin{align*}
H^{s,p}=J_s(L^p(\Rn)), \qquad \text{ with }\, J_s(\varphi)=\mathscr{F}^{-1}((1+|\xi|^2)^{-\frac s2} \widehat{\varphi}),
\end{align*}
For any $f\in H^{s,p}$, we use the following notation
\[ \begin{split}
\|f\|_{H^{s,p}}&=\|\mathscr{F}^{-1}((1+\xii^2)^{\frac{s}2}\,\hat f)\|_{L^p},\\
\|f\|_{\dot H^{s,p}}&=\|\mathscr{F}^{-1}(\xii^{s}\,\hat f)\|_{L^p}.
\end{split} \]
Then, for any $s\geq0$ (see~\cite{Stein70})
\[ \|f\|_{H^{s,p}} \approx \|f\|_{L^p}+\|f\|_{\dot H^{s,p}}. \]
\end{notation}
When~$m\in\mathbb{N}$ and $p\in(1,\infty)$,
\[ H^{m,p}=W^{m,p}=\{u\in L^p: \ \partial_x^\alpha u \in L^p:\,|\alpha|\leq m\},\]
the usual Sobolev space of order~$m$, and $\|\cdot\|_{H^{m,p}}$ is equivalent to $\|\cdot\|_{W^{m,p}}$ (see Theorem 3 in V, \textsection 3.3 in~\cite{Stein70}). If $m$ is an even positive integer, then $W^{m,p}\hookrightarrow H^{m,p}$ when $p=1,\infty$ (see V, \textsection 6.6 in~\cite{Stein70}).
\begin{notation}[see~\cite{Ho}]\label{not:Mpq}
Let $1\leq p\leq q\leq\infty$. Then $L_p^q=L_p^q(\mathbb{R}^n)$ denotes the space of tempered distributions $T\in \mathcal S'$ such that~$T\ast f\in L^q$ for any~$f\in \mathcal S$, and
\[ \| T \ast f\|_{L^q} \leq C \|f\|_{L^p}, \]
for all $f\in\mathcal S$ with a constant $C$, which is independent of $f$. In this case, the operator~$T\ast$ is extended by density from~$\mathcal S$ to~$L^p$ when $p\in[1,\infty)$. 
%
%
The notation $M_p^q=M_p^q(\mathbb{R}^n)=\mathscr{F}(L_p^q)$ denotes the set of Fourier transforms $\hat{T}$ of distributions $T \in L_p^q$, equipped with the norm 
\[ \|m\|_{M_p^q}=\sup\big\{\|\mathscr{F}^{-1}(m\mathscr{F}(f))\|_{L^q}:f\in \mathcal{S}, \|f\|_{L^p}=1\big\}.\]
For $p=\infty$, $L_\infty^\infty$ is the space of bounded measures and $M_\infty^\infty$ is the space of their Fourier transforms. The elements in $M_p^q$ are called  multipliers of type $(p,q)$.
\end{notation}
\begin{notation}\label{not:CN1}
Let $K\subset\R^n$, compact, $v\in\mathcal C^N(K)$, and $w\in\mathcal C^{N,1}(K)$, that is, the derivatives of order $N$ of $w$ are Lipschitz-continuous. We use the notation
\[ \begin{split}
\|v\|_{\mathcal C^N(K)}
    & = \max_{\stackrel{x\in K}{|\alpha|\leq N}} |\partial_x^\alpha v(x)|, \\
\|w\|_{\mathcal C^{N,1}(K)}
    & = \|w\|_{\mathcal C^N(K)}+ \max_{|\alpha|=N}\,\sup_{\stackrel{x,y\in K}{x\neq y}} \frac{|\partial_x^\alpha(w(x)-w(y))|}{|x-y|}\,.
\end{split} \]
\end{notation}
\begin{notation}
Let $f$ be a scalar, vector-valued, or matrix-valued function; then $\nabla f$ denotes its gradient vector, matrix or tensor and $H_f$ denotes its Hessian matrix or tensor. In particular, for scalar functions $f$, $\nabla H_f=(\partial_{x_i}\partial_{x_j}\partial_{x_k}f)_{i,j,k}$. We put $|\nabla f|=\max_j|\partial_{x_j}f|$, $|H_f|=\max_{j,k}|\partial_{x_j}\partial_{x_k}f|$, $|\nabla H_f|=\max_{i,j,k}|\partial_{x_i}\partial_{x_j}\partial_{x_k}f|$, \ldots
\end{notation}
%
%
%

\section{Proof of Theorem \ref{thm:hyp}} \label{sec:lin est}
In this section we want to provide $L^p-L^q$ estimates for $1\leq p\leq q\leq \infty$ the solution of model \eqref{eq:CP:linear}. 
After applying the Fourier transform to \eqref{eq:CP:linear} in space variable only, we may write
\begin{equation*}
\begin{cases}
\widehat{u}_{tt} + a(\xi)^2\widehat{u}+ b(\xi)\widehat{u}_t  =0, & t\geq 0,\, x \in \Rn,\\
\widehat{u}(0,x)=0,\,\widehat{u}_t(0,x)=\widehat{u_1}(x),
\end{cases}
\end{equation*}
Then the characteristic roots of the full symbol are
\begin{equation}\label{eq:char roots}
\lambda_{\pm}(\xi)= 
-\frac{b(\xi)}{2} \pm i w(\xi),
\end{equation}
where
\begin{equation}\label{eq:phase func}
w(\xi)=a(\xi)\sqrt{ 1- b^2(\xi)/ (4a^2(\xi))},
\end{equation}
due to the assumption $\delta \geq k/2.$
Let $E(t)$ be the fundamental solution to~\eqref{eq:CP:linear}, that is, the solution to
\begin{equation}\label{eq:K}
\begin{cases}
E_{tt}+ A E + BE_t=0, & t\geq 0,\\
E(0)=0,\,E_t(0)=\delta_0.
\end{cases}
\end{equation}
Thus, we may write the fundamental solution as
\begin{equation}\label{eq:fund sol}
\widehat{E}(t,\cdot)= \mathscr{F}^{-1}\left( t\,e^{-\frac{b(\xi)}2t} \sinc (w(\xi) t)\right).
\end{equation}
In order to prove Theorem \ref{thm:hyp}, we first prepare a result for a more general multiplier, which depends on a generic parameter $t$.

Consider a generic parameter $t$ and the $t$-dependent multiplier
\begin{equation}\label{eq:general}
m(t,\xi)=\xii^{-sk}\, g(t,\xi)\,e^{\pm i w(t,\xi)},
\end{equation}
where $g(t,\cdot)\in\mathcal C^{n}(\R^n\setminus\{0\})$ and $w(t,\cdot)\in\mathcal C^{n+1,1}(\R^n\setminus\{0\})$ satisfy the following assumptions, uniformly with respect to the parameter $t$,
\begin{equation}\label{eq:assumptions}\begin{split}
|\partial_\xi^\alpha g(t,\xi)|
    & \leq C\,\xii^{-|\alpha|}, \qquad |\alpha|=0,\ldots,n,\\
|\partial_\xi^\alpha w(t,\xi)|
    & \leq C\,\xii^{k-|\alpha|}, \qquad |\alpha|=0,\ldots,n+1,\\
w(t,\xi)
    & =a(\xi)+\textit{o}(\xii^k), \qquad \text{as~$\xii\to\infty$,}
\end{split} \end{equation}
where $a\in\mathcal C^{n+1}(\R^n\setminus\{0\})\cap \mathcal C^{n+1,1}(S^{n-1})$ is homogeneous of degree $k$ ($k>0$ with $k\neq1$) and $\det H_{a(\xi)}\neq0$, for any $\xi\neq0$.

Then we have the following.

\begin{lemma}\label{lem:high}
Assume~\eqref{eq:assumptions}. Let~$1\leq p\leq q\leq \infty$ and $s>-n/2$. Assume that $d(p,q)<s$. Then there exists~$C>0$, independent of the parameter~$t$, such that
\[ \|(1-\chi)\,\xii^{-sk}\,g(t,\xi)\,e^{\pm iw(t,\xi)} \|_{M_p^q} \leq C,\]
for any~$t\geq0$.
\end{lemma}

The assumption $s>-n/2$ is due to the fact that $d(p,q)>-n/2$, for any $p,q$ and $k$ (see~\eqref{eq: d(p,q)}). Lemma \ref{lem:high} is similar to Lemma 3 in \cite{DAL25b}, but here we consider the couple $(p,q)$ not restricted in the conjugate line.
 Moreover, the proof of Lemma~\ref{lem:high} is based on the version of Littman’s lemma \ref{lem:Littman}, where the phase function $w$ is assumed to have only $\mathcal C^{n+1,1}$ regularity and the amplitude $v$ is assumed to have $\mathcal C^n$ regularity. As a consequence of these assumptions, we produce a decay $\tau^{-\frac{n}2}\log\tau$ for $\tau\gg1$. The logarithmic loss has no influence on our estimates, since we use this decay on dyadic level, where the convergence of a series is independent of a loss smaller than any power (see the proof of Lemma~\ref{lem:high}). Therefore, it is a natural choice for us to reduce the regularity of the phase function to $\mathcal C^{n+1,1}$ and the regularity of the amplitude function to $\mathcal C_c^n$.
We postpone the proof of Lemma \ref{lem:high} to \textsection \ref{sec:osc}.\\

We are now ready to prove Theorem \ref{thm:hyp}.

\begin{proof}[Proof of Theorem \ref{thm:hyp}]
We define
\begin{equation}\label{eq:wtxi}
w(t,\xi)=tw(t^{-\frac1k}\xi) = a(\xi)\sqrt{ 1- t^2b^2(t^{-\frac1k}\xi)/ (4a^2(\xi))}.
\end{equation}
First of all, we notice that $\chi\,|\xi|^\ell\,\sinc w(t,\xi)$ is in $M_p^q$ for any $\chi\in\mathcal C_c^\infty$ and $1 \leq p \leq q \leq \infty$. Indeed, since
\[|\partial_\xi^{\alpha} \sinc w(t,\xi) | \leq C w(t,\xi)^2\,\xii^{-|\alpha|} \leq C_1\, \xii^{2k-|\alpha|},\]
for $|\alpha|\geq 1$, we get that
\[ \chi\,\xii^{\ell}\,\sinc w(t,\xi)\in M_q^q,\]
by applying Theorem~\ref{Thm_Mik-Hor} if $q \in (1,\infty)$ or Lemma \ref{lem:HF} in the case $q=1,\infty$ (and $\chi m\in M_q^q$ implies that $\chi m\in M_p^q$ for any $1\leq p\leq q$, for $\chi\in\mathcal C_c^\infty$).

Then, combining the previous result with Lemma~\ref{lem:high} with $s=1-\ell/k$, we get
\begin{align}\label{eq:est_osc_mult2}
 \|\xii^\ell\,\sinc w(t,\xi) \|_{M_p^q}\leq C,
\end{align} 
provided that $d(p,q)+\ell/k<1$. We stress that the crucial point is that~\eqref{eq:est_osc_mult2} is independent of $t$ since the estimates of Lemma~\ref{lem:high} are independent of the parameter $t$, whatever its role as parameter originally was, provided that \eqref{eq:assumptions} holds (independently of $t$), and~\eqref{eq:assumptions} holds (independently of $t$) with $g=1/w(t,\xi)$ when $w(t,\xi)$ is as in \eqref{eq:wtxi}. In particular, this is due to the fact that for any $|\alpha|=1,\ldots,n+1$, we get
\[ \frac{t}{|a(\xi)|}\,|\partial_\xi^\alpha b(t^{-\frac1k}\xi)|\leq C\xii^{-|\alpha|}\,\frac{|t^{-\frac1k}\xi|^{2\delta-k}}{1+|t^{-\frac1k}\xi|^{2\delta+\varepsilon}}\leq C\xii^{-|\alpha|}. \]

We have to consider the following quantity
\begin{align*}
    \| (-\Delta)^{\frac{\ell}2}\,u(t,\cdot)\|_{L^q}= \|(-\Delta)^{\frac{\ell}2}\,E(t)\ast u_1 \|_{L^q}.
\end{align*}

If $d(p,q)+\ell/k< 1$, using the change of variable $\xi t^{\frac1k}\mapsto \xi$, so that $a(\xi)t\mapsto a(\xi)$ in view of the $k$-homogeneity assumption, we get
\[ \begin{split}
\|(-\Delta)^{\frac{\ell}2}\,E(t)\ast u_1 \|_{L^q}
& \leq \|e^{-b(\xi)\frac{t}2}\|_{M_q^q}\,t\,\|\xii^\ell\,\sinc (w(\xi)t)\|_{M_p^q}\,\|u_1\|_{L^p}\\
& \leq \|e^{-b(\xi)\frac{t}2}\|_{M_q^q}\,\|\xii^\ell\,\sinc (w(t,\xi)\|_{M_p^q}\,\|u_1\|_{L^p}\,t^{1-\frac{n}k\left(\frac1p-\frac1q\right)-\frac\ell{k}}\,\\
    & \leq C\,t^{1-\frac{n}k\left(\frac1p-\frac1q\right)- \frac\ell{k} }\,\,\|u_1\|_{L^p},
\end{split}\]

where we used \eqref{eq:est_osc_mult2} and Lemma~\ref{lem:high0}, in particular, \eqref{eq:hest}.

We now assume $d(p,q)+\ell/k\geq 1$. Then
\[ \begin{split}
\|(-\Delta)^{\frac{\ell}2}\,E(t)\ast u_1 \|_{L^q}
& \leq\|\mathscr{F}^{-1}(|\xi|^{k(s+\ell/k-1)}e^{-b(\xi)\frac{t}2}) \ast u_1\|_{L^p}\,\times \\
&\qquad \qquad\times t\,\|\xii^\ell\,|\xi|^{-k(s+\ell/k-1)}\,\sinc (w(\xi)t)\|_{M_p^q}\\
& \leq \|\mathscr{F}^{-1}(|\xi|^{k(s+\ell/k-1)}e^{-b(\xi)\frac{t}2}) \ast u_1\|_{L^p}\,t^{s-\frac{n}k\left(\frac1p-\frac1q\right)}.
\end{split}\]
In order to estimate the exponential term, 
%
%
we use Lemma \ref{lem:high0} to get
\begin{align*}
\|\mathscr{F}^{-1}(|\xi|^{k(s+\ell/k-1)}&e^{-b(\xi)\frac{t}2}) \ast u_1\|_{L^p}\\
& \qquad\leq e^{-ct}   \|u_1\|_{\dot{H}^{k(s+\ell/k-1),p}}                        + (1+t)^{-\frac \ell{2\delta}- \frac k{2\delta}(s-1)}\|u_1\|_{L^p}.
\end{align*}
This concludes the proof of Theorem~\ref{thm:hyp}.
\end{proof}

%

\subsection{Proof of Lemma~\ref{lem:high0}}
In this subsection, we prove Lemma~\ref{lem:high0}, using Lemma \ref{lem:LF} and \ref{lem:HF}, stated in \textsection\ref{sec:App}.
\begin{proof}

Estimate~\eqref{eq:hestlow} directly follows by using Lemma \ref{lem:LF} with $\theta=2\delta$ and $\eta=s \delta$. Let us prove~\eqref{eq:hesthigh}. We first prepare suitable estimates on the derivatives of our multipliers.

If $\xi\in\supp(1-\chi)$, we may estimate
\begin{equation}\label{eq:hhigh}
|\partial_\xi^\alpha e^{-b(\xi)t} | \leq C\,\xii^{-\varepsilon-|\alpha|}\,e^{-ct},
\end{equation}
for any $|\alpha|=1,\ldots,n+1$. 
%
 Having in mind~\eqref{eq:hhigh}, let $f(t,\xi)=e^{ct}\,(1-\chi)\,e^{-t b(\xi)}$. Then, we may apply Lemma \ref{lem:HF}, getting $\|f(t,\cdot)\|_{M_1^1}\leq C$, with $C>0$ independent of $t$. Therefore,
\[ \|(1-\chi)\,e^{-tb}\|_{M_1^1} = e^{-ct}\,\|f(t,\cdot)\|_{M_1^1} \leq C\,e^{-ct}, \]
with $C,c>0$ independent of $t$, i.e., we proved~\eqref{eq:hesthigh}. This concludes the proof.
\end{proof}
\subsection{Proof of Lemma~\ref{lem:high}}\label{sec:osc}
In this subsection, we prove Lemma~\ref{lem:high}, interpolating  $L^2-L^2$, $L^1-L^\infty$ and $L^1-L^1$ estimates at dyadic level. In particular, we use Corollary~\ref{cor:Littman} for $L^1-L^\infty$ estimates. 

%
\begin{proof}
The dependence on $t$ could be omitted in this proof to make it more readable, still we prefer to keep it to emphasize how the estimates are, indeed, independent of $t$.

Let~$m(t,\xi)$ be as in~\eqref{eq:general}, for brevity. We consider a dyadic partition of unity~$\{\psi_j\}_{j\in \mathbb{Z}}$. 
We may assume without loss of generality that $\chi=1$ in $B(0,2^{j_0})$, so that
\[ (1-\chi)\psi_j =0, \qquad \forall j\leq j_0-1. \]
Indeed, for any $\tilde\chi\in\mathcal C_c^\infty(\R^n\setminus\{0\})$, it holds $(1-\chi)\tilde\chi \,m(t,\xi)\in\mathcal C_c^{n+1}\subset \mathscr{F}(L^1)$, with uniformly bounded derivatives with respect to $t$.

We aim to estimate
\[ \|(1-\chi)\,m(t,\xi)\|_{M_p^{p'}} \leq C\,\sum_{j=j_0}^{+\infty} \|\psi_j(\xi)\,m(t,\xi)\|_{M_p^{p'}}, \]
where we already used that $(1-\chi)\psi_j=0$ for $j\leq j_0-1$ and we put $C=\|1-\chi\|_{M_q^q}$ for $j\geq j_0$. Let $j\geq j_0$ in the following. It is clear that
\begin{equation}\label{eq:M22 estimate}
\| \psi_j m(t,\cdot)\|_{M_2^2} = \| \psi_j m(t,\cdot)\|_{L^\infty} = \| \psi_j \xii^{-sk}\,g(t,\cdot)\|_{L^\infty} \approx 2^{-kjs}.
\end{equation}
Now we want to obtain an estimate of~$\| \psi_j m(t,\xi)\|_{M_1^\infty}$, using \eqref{eq:Old Littman}.
By the change of variables~$2^{-j}\xi= \eta$, we get
\[\begin{split}
\| \psi_j m(t,\xi)\|_{M_1^\infty}
    & = \left\|\mathscr{F}^{-1}_{\xi \to x}\left( m(t,\xi) \psi_j(\xi) \right) \right\|_{L^\infty} \\
    & =2^{jn}\left\|\mathscr{F}^{-1}_{\eta \to x}\left( m(t,2^j\eta) \psi(\eta) \right) \right\|_{L^\infty} \\
    & =2^{jn-kjs}\left\|\mathscr{F}^{-1}_{\eta \to x}\left( |\eta|^{-sk}\,g_j(t,\eta)\,e^{\pm i 2^{kj}w_j(t, \eta)} \psi(\eta) \right) \right\|_{L^\infty},
\end{split}\]
where
\[
w_j(t,\eta)= 2^{-kj} w(t,2^j \eta), \qquad g_j(t,\eta)=g(t,2^j\eta).
\]
In particular, in view of~\eqref{eq:assumptions},
\begin{equation}\label{eq:assumptionsdyadic}
|\partial_\eta^\alpha g_j(t,\eta)|  \leq C\,|\eta|^{-|\alpha|}, \qquad |\partial_\eta^\alpha w_j(t,\eta)|  \leq C\,|\eta|^{-|\alpha|},
\end{equation}
with $C>0$ independent of $t$ and $j$. Thanks to~\eqref{eq:assumptions}, we get
\[ \lim_{j\to\infty} |w_j(t,\eta) - 2^{-kj} a(2^j\eta)|=0, \]
uniformly with respect to $|\eta|\in[1/2,2]$; hence, for a sufficiently large $j_0$, the Hessian matrix of the phase function~$w_j$ is nonsingular, that is,
\begin{align*}
|\det H_{w_j(t,\eta)}| \geq c>0\,\qquad \forall |\eta| \in [1/2,2].
\end{align*}
Then, applying~\eqref{eq:New Littman} with $\tau=2^{kj}$, we conclude that
\begin{align}\label{eq:M1inf estimate}
\| \psi_j m(t,\xi)\|_{M_1^\infty} \leq C\, 2^{jn-kjs} (1+ 2^{kj})^{-\frac n2}\, \log(e+2^{kj}) \approx C\,j\, 2^{jn\left(1-\frac{k}2\right)-kjs},
\end{align}
where, in the last step, we used $j\geq j_0\geq 0$. 

We stress that the application of~\eqref{eq:New Littman} is possible due to the fact that, after the change of variable $\eta=2^{-j}\xi$, all the dyadic annuli reduce to the same (compact) annulus $|\eta|\in[1/2,2]$. As a consequence, on the one hand, $D_n=2^n$ independently on $j$; on the other hand, the $\mathcal C^n$ estimates of the amplitude function $g_j$ and the $\mathcal C^{n-1}$ estimates of $H_{\omega_j}$, which come into play into $M_n$, are uniform with respect to~$j\geq j_0$, thanks to~\eqref{eq:assumptionsdyadic}.

Using~\eqref{eq:M22 estimate},~\eqref{eq:M1inf estimate} and Riesz-Thorin interpolation theorem, we get
\begin{align}\label{eq:conj_line estimate}
\| \psi_j m(t,\xi)\|_{M_{p_0}^{p_0'}} \leq C\, j^{\frac2{p_0}-1}\, 2^{jn(2-k)\left(\frac1{p_0}-\frac12\right)-kjs}, \qquad \forall p_0\in[1,2].
\end{align}
Finally we want to obtain the~$L^1-L^1$ estimate. Let~$N=[n/2]+1$. Since
\[\begin{split}
&\big|\partial_\xi^{\alpha} g(t,\xi)\big| \leq C |\xi|^{-|\alpha|},\\
&\big|\partial_\xi^{\alpha} e^{\pm iw(t,\xi)}\big| \leq C |\xi|^{|\alpha|(k-1)},
\end{split}\]
for any $|\alpha|\leq N$, and
\[ \partial_x^\alpha \psi_j(\xi) = 2^{-j|\alpha|}\,(\partial_\xi^\alpha\psi)(2^{-j}\xi), \]
we get
\[ \big|  \partial_\xi^{\alpha}(\psi_j m(t,\xi)) \big| \leq 2^{j(-ks-|\alpha|)} (1+2^{jk|\alpha|})\approx 2^{j(-ks+|\alpha|(k-1))}, \]
for any $|\alpha|\leq N$ and $j\geq j_0$, where we used $j_0\geq0$ in the last step. Therefore,
\begin{align*}
\| \partial_\xi^{\alpha}(\psi_j m(t, \cdot))\|_{L^2} \leq C\, 2^{j\left(-ks+\frac{n}{2}+|\alpha|(k-1) \right)},
\end{align*}
for any $|\alpha|\leq N$. Using Theorem~\ref{Thm_Bernestein in}, we derive
\begin{equation}\label{eq:M11 estimate}
\begin{split}
\| \psi_j m(t, \cdot)\|_{M_{1}^{1}}
    & \leq C\| \psi_j m(t, \cdot)\|_{L^2}^{1-\frac n{2N}} \sum_{|\alpha|=N} \| \partial_\xi^\alpha(\psi_j m(t, \cdot))\|_{L^2}^\frac n{2N} \\
    & \leq C 2^{kj\frac n2-kjs}.
\end{split}\end{equation}
Using~\eqref{eq:conj_line estimate}, ~\eqref{eq:M11 estimate} and Riesz-Thorin interpolation theorem, 
we conclude that
\[
\| \psi_j m(t, \cdot)\|_{M_{p}^{q}} \leq C 2^{jn(2-k)\left(\frac1{p_0}-\frac12\right)(1-\theta)+kj\frac{n}2\theta-kjs}= C\, j^{\frac 1p-\frac 1q}\,2^{jk(d(p,q)-s)}.
\]
Assuming $d(p,q)<s$, we find
\[ \|(1-\chi)\,m(t,\xi)\|_{M_p^q} \leq C\,\sum_{j=j_0}^{+\infty}\,j^{\frac 1p-\frac 1q}\, 2^{jk(d(p,q)-s)}<\infty\,. \]
This concludes the proof.

\end{proof}

\section{Proof of Theorems \ref{thm:glob_ex} and \ref{thm:glob_ex_fuj}} \label{sec:nonlinear}


Let $E(t,\cdot)$ be the fundamental solution to~\eqref{eq:CP:nonlinear} with $f=0$, that is, $u(t,\cdot)=E(t,\cdot)\ast u_1$. By Duhamel's principle, for a suitable space~$X$ such that $E(t,\cdot)\ast u_1$ is in~$X$, let us consider the integral operator~$N:X\to X$ defined by
\[ Nu(t,\cdot) = \int_0^t E(t-\tau)\ast B f(u(\tau,\cdot))\,d\tau. \]
Then a fixed point for the operator
\[ \Lambda: u(t,x) \mapsto \Lambda u= E(t,\cdot)\ast u_1 + Nu(t,x) \]
in~$X$ gives a solution to~\eqref{eq:CP:nonlinear} in~$X$.

Global existence of (uniquely determined) small data solutions to~\eqref{eq:CP:nonlinear} in $X$ follows by standard contraction argument, if we prove that~$\Lambda$ is a contraction on~$X$ for sufficiently small data~$u_1$ in some space.

We fix a space~$X$ with a suitable norm such that
\begin{equation}\label{eq:Xdata}
\|E(t,\cdot)\ast u_1\|_X \leq C\,\|u_1\|_{\mathcal A},
\end{equation}
for some $C$, independent of $u_1$ and we prove the following a priori estimate:
\begin{equation}\label{eq:contr}
\|Nu-Nv\|_X \leq C \|u-u\|_X\left(\|u\|_X^{\alpha-1}+\|u\|_X^{\alpha-1}\right).
\end{equation}
One of the tools we will use is the following classical lemma.
\begin{lemma}\label{lem}
Let~$\nu >-1$ and~$\mu \in \mathbb{R}$. Then we have
\begin{equation}
\int_{0}^{t}(t-s)^{\nu}\,(1+s)^{\mu}\,ds\lesssim
\begin{cases}
(1+t)^{\nu} &   \mbox{ if } \mu <-1,\\
(1+t)^{\nu}\,\log(e+t), &  \mbox{ if } \mu =-1,\\
(1+t)^{1+\nu+\mu} &  \mbox{ if }  \mu >-1.
\end{cases}
\end{equation}
%
%
Moreover, the estimate is also valid if~$(t-s)^{\nu}$ is replaced by~$(1+t-s)^{\nu}$ in the integral.
\end{lemma}
Lemma~\ref{lem} has been proved by many authors, see, for instance,~\cite{Segal}.

\begin{proof}[Proof of Theorem \ref{thm:glob_ex}]
Let us define the space

\[ X\doteq\left\{u\in L^\infty([0,\infty), L^{\alpha+1}) : \ \|u\|_X<\infty\right\}, \]
where the norm~$\|\cdot\|_X$ is defined by
\begin{equation}\label{eq:Xnorm_1}
\|u\|_X \doteq \sup_{t\geq0} \left((1+t)^{\delta(\alpha)}\,\|u(t,\cdot)\|_{L^{\alpha +1}}
\right).
\end{equation}
where
\[\delta(\alpha)= \frac n2\left(\frac 1{\delta} +1-\frac k{2\delta}\right)\frac {\alpha-1}{\alpha+1}+ \left(1-\frac k{2\delta} \right)-\varepsilon\left(1-\frac k{2\delta} \right) ,\]
for a sufficiently small $\varepsilon>0$ such that $\varepsilon\leq 1-d(1+1/\alpha, 1+\alpha)$.

Using \eqref{eq:pert_fr} with $\ell=2\delta $ and $d(1+1/\alpha, 1+\alpha) \geq 1-2\delta/k$, we get
\[\begin{split}
\|u^\lin(t,\cdot)\|_{L^{\alpha+1}}
    & = \|(-\Delta)^{-\delta}(-\Delta)^{\delta}\,u^\lin(t,\cdot)\|_{L^{\alpha+1}} \\
    & \quad\leq C_1\,t^{-\delta(\alpha)}\,\left(\|u_1\|_{ L^{1+\frac 1 \alpha}}+\|u_1\|_{\dot H^{-2\delta, 1+\frac 1 \alpha}}\right)\leq C_1\,t^{-\delta(\alpha)}\varepsilon_0.
\end{split}
\]
On the other hand for $t \leq 1$, since $\alpha \leq (n+k)/(n-k)_+$ and $d< 1$, using \eqref{eq:pert_fr} with $\ell=0$, we get
\begin{equation*}
\|u^\lin(t,\cdot)\|_{L^{\alpha+1}} \leq C\, \|u_1\|_{L^{1+\frac 1\alpha}}.
\end{equation*}
Then, we immediately derive~\eqref{eq:Xdata}, so we may now prove~\eqref{eq:contr} for any~$u,v\in X$. \\

From \eqref{eq:Xnorm_1}, for any~$u\in X$ and for any~$\tau\geq0$, it holds
\begin{align}
\label{eq:estuX}
\| u(\tau,\cdot)\|_{L^{\alpha +1}}
    & \leq (1+\tau)^{-\delta(\alpha)} \|u\|_X,
\end{align}

Using Lipschitz-condition on $f$, H\"older inequality and \eqref{eq:estuX}, we find that
\begin{align}\label{eq:estfX}
\begin{split}
 \|(f(u)-f(v))(\tau,\cdot)\|_{L^{1+\frac1\alpha}}
   & \lesssim \|(u-v)\left(|u|^{\alpha-1}+|v|^{\alpha-1}\right)(\tau,\cdot)\|_{L^{1+\frac1\alpha}}\\
   & \quad \lesssim \|(u-v)(\tau,\cdot)\|_{L^{\alpha+1}}\left(\|u(s,\cdot)\|_{L^{\alpha+1}}^{\alpha-1}+\|v(\tau,\cdot)\|_{L^{ \alpha+1}}^{\alpha-1}\right)\\
   &\quad \lesssim (1+\tau)^{-\alpha \delta(\alpha)}\,\|u-v\|_X\,\big(\|u\|_X^{\alpha-1}+\|v\|_X^{\alpha-1}\big).
\end{split}
\end{align}
%
We consider the case $t \geq 2$, since $t<2$ is trivial.
Then, applying~Corollary \ref{cor:hyp} with $s=d+\varepsilon$, $d < 1$ and \eqref{eq:estfX}, we derive 
\[ \begin{split}
&\| (Nu-Nv)(t,\cdot)\|_{L^{\alpha+1}}\\
    & \qquad\lesssim \int_0^{t-1} (t-\tau)^{-\delta(\alpha)} \|(f(u)-f(v))(\tau,\cdot)\|_{L^{1+\frac1\alpha}}\,d\tau\\
    &\qquad \qquad + \int_{t-1}^t (t-\tau)^{-\frac n2 \frac{\alpha-1}{\alpha+1}+\varepsilon } \|(f(u)-f(v))(\tau,\cdot)\|_{L^{1+\frac1\alpha}}\,d\tau\\
   & \qquad \lesssim \|u-v\|_X\,\big(\|u\|_X^{\alpha-1}+\|v\|_X^{\alpha-1}\big)\,\int_0^{t-1} (t-\tau)^{-\delta(\alpha)}\,(1+\tau)^{-\alpha \,\delta(\alpha)} \,d\tau\\
   &\qquad \qquad + (1+t)^{-\alpha \,\delta(\alpha)}  \|u-v\|_X\,\big(\|u\|_X^{\alpha-1}+\|v\|_X^{\alpha-1}\big)\,,
\end{split}
\]
since $\alpha< (n+2)/(n-2)_+$. From \eqref{eq: strauss exp}, we have
\[ \frac1\alpha<   \delta(\alpha) < 1, \]
so that, using Lemma \ref{lem}, we get
\[\int_0^{t-1} (t-\tau)^{-\delta(\alpha)}\,(1+\tau)^{-\alpha \,\delta(\alpha)} \,d\tau \leq C(1+t)^{-\delta(\alpha)}.\]
This concludes the proof.
\end{proof}

We now move to the proof of Theorem \ref{thm:glob_ex_fuj}.

\begin{proof}[Proof of Theorem \ref{thm:glob_ex_fuj}]
Let us define the space

\[ X\doteq\left\{u\in L^\infty([0,\infty), L^{\alpha}) : \ \|u\|_X<\infty\right\}, \]
where the norm~$\|\cdot\|_X$ is defined by
\begin{equation}\label{eq:Xnorm_1_fuj}
\|u\|_X \doteq \sup_{t\geq0} \left((1+t)^{\delta(\alpha)}\,\|u(t,\cdot)\|_{L^{\alpha}}
\right),
\end{equation}
where
\[  \delta(\alpha)=\frac n{2\delta}  \left(1-\frac 1{\alpha} \right)-\frac n{\alpha}\left(1-\frac k{2\delta} \right) +\left( \frac n2+1 \right)\left(1-\frac k{2\delta} \right)-\ep \left(1-\frac k{2\delta}\right) \]
for a sufficiently small $\varepsilon>0$  such that $\varepsilon\leq 1-d(1,\alpha) $.
For $t \geq 1$, using \eqref{eq:pert_fr} with $\ell=2$ and $d=d(1,\alpha) \geq 1-2\delta/k$, we get
\[\begin{split}
\|u^\lin(t,\cdot)\|_{L^{\alpha}}
    & = \|(-\Delta)^{-\delta}(-\Delta)^{\delta}\,u^\lin(t,\cdot)\|_{L^{\alpha}} \\
    & \leq C \,t^{-\delta(\alpha)}\,\|(-\Delta)^{-\delta}\,u_1\|_{L^{1}}\\
    & \leq C_1\,t^{-\delta(\alpha)}\,\left(\|u_1\|_{ L^{1}}+\|u_1\|_{\dot H^{-2\delta, 1}}\right)\leq C_1\,t^{-\delta(\alpha)}\varepsilon_0.
\end{split}
\]
On the other hand for $t \leq 1$, since $\alpha \leq n/(n-k)_+$, we get
\begin{equation*}
\|u^\lin(t,\cdot)\|_{L^{\alpha}} \leq C\,\|u_1\|_{L^1}.
\end{equation*}
Then, we immediately derive~\eqref{eq:Xdata}, so we may now prove~\eqref{eq:contr} for any~$u,v\in X$. \\

From \eqref{eq:Xnorm_1_fuj}, for any~$u\in X$ and for any~$\tau\geq0$, it holds
\begin{align}
\label{eq:estuX_fuj}
\| u(\tau,\cdot)\|_{L^{\alpha}}
    & \leq (1+\tau)^{-\delta(\alpha)} \|v\|_X,
\end{align}
As in the proof of Theorem \ref{thm:glob_ex}, 
let $t \geq 2$.
Using Lipschitz-condition on $f$, H\"older inequality and \eqref{eq:estuX_fuj}, we find that
\begin{align}\label{eq:estfX_fuj}
\begin{split}
\|(f(u)-f(v))(\tau,\cdot)\|_{L^1}
   & \lesssim \|(u-v)\left(|u|^{\alpha-1}+|v|^{\alpha-1}\right)(\tau,\cdot)\|_{L^{1}}\\
   & \quad\lesssim \|(u-v)(\tau,\cdot)\|_{L^{\alpha}}\left(\|u(s,\cdot)\|_{L^{\alpha}}^{\alpha-1}+\|v(\tau,\cdot)\|_{L^{ \alpha}}^{\alpha-1}\right)\\
   &\quad \lesssim (1+\tau)^{-\alpha \delta(\alpha)}\,\|u-v\|_X\,\big(\|u\|_X^{\alpha-1}+\|v\|_X^{\alpha-1}\big).
\end{split}
\end{align}
Then, applying~Corollary \ref{cor:hyp}, and \eqref{eq:estfX_fuj}, since $d <1$, we derive
\[ \begin{split}
&\| (Nu-Nv)(t,\cdot)\|_{L^{\alpha}}\\
    &\qquad \lesssim \int_0^{t-1} (t-\tau)^{-\delta(\alpha)}\|(f(u)-f(v))(\tau,\cdot)\|_{L^1}\,d\tau\\
  &\qquad  \qquad +\int_{t-1}^t (t-\tau)^{-n\left(\frac 12 - \frac 1\alpha \right)+\varepsilon}\|(f(u)-f(v))(\tau,\cdot)\|_{L^1}\,d\tau\\  
&\qquad     \lesssim \|u-v\|_X\,\big(\|u\|_X^{\alpha}+\|v\|_X^{\alpha-1}\big)\,\int_0^t (t-\tau)^{-\delta(\alpha)}\,(1+\tau)^{-\alpha \,\delta(\alpha)} \,d\tau\\
&\qquad \qquad +
(1+t)^{-\alpha \,\delta(\alpha)} \|u-v\|_X\,\big(\|u\|_X^{\alpha}+\|v\|_X^{\alpha-1}\big)\,,
\end{split}
\]
where we used $\alpha < 2n/(n-2)_+.$
From \eqref{eq: fuj exp}, we have
\[ \frac1\alpha<   \delta(\alpha) < 1, \]
so that, using Lemma \ref{lem}, we get
\[\int_0^{t-1} (t-\tau)^{-\delta(\alpha)}\,(1+\tau)^{-\alpha \,\delta(\alpha)} \,d\tau \leq C(1+t)^{-\delta(\alpha)}.\]
This concludes the proof.
\end{proof}

\appendix
\section{Tools from Fourier Analysis} \label{sec:App}

We first state the well-know Mikhlin-H\"ormander multiplier theorem.
\begin{theorem}\label{Thm_Mik-Hor}
Let~$q\in (1,\infty)$ and~$k = [n/2] + 1$. Suppose that~$m \in \mathcal{C}^k(\Rn\setminus \{0\})$ and
\begin{align*}
\big| \partial^\beta_\xi m(\xi)\big|\leq C|\xi|^{-|\beta|},\quad |\beta|\leq k.
\end{align*}
Then~$m \in M_q^q$.
\end{theorem} 

The following result, usually called Berstein's inequality, is useful to prove $L^1-L^1$ estimates, see~\cite[Theorem 1.2]{Sjostrand}.
\begin{theorem}\label{Thm_Bernestein in}
Let~$N > n/2$ and assume~$m\in H^N$. Then~$\mathscr{F}^{-1}m \in L^1$ and there exists a constant~$C>0$ such that
\begin{align*}
\|\mathscr{F}^{-1}m\|_{L^1}\leq C\|m\|_{L^2}^{1-\frac n{2N}}\|D^N m\|_{L^2}^{\frac n{2N}}. 
\end{align*}
\end{theorem}

For the estimates of the dissipative multiplier at low frequencies, we state the following lemma.

\begin{lemma}[see~Lemma 3.2 in~\cite{DAE25}]\label{lem:LF}
Let $g\in\mathcal C^{n+1}(\R^n\setminus\{0\})$ verifies
\[ b(\xi)\geq c\,\xii^{\theta},\quad |\partial_\xi^\gamma b(\xi)|\leq C\,\xii^{\theta-|\gamma|},\qquad \text{for $1\leq|\gamma|\leq n+1$,} \]
for some~$c>0$, $C>0$, $\theta>0$, and let $g\in\mathcal C^{n+1}(\R^n\setminus\{0\})$ verify
\[ |\partial_\xi^\gamma g(\xi)|\leq C\,\xii^{\eta-|\gamma|},\quad \qquad \text{for $0\leq|\gamma|\leq n+1$,}\]
for some $\eta>-n$. Then
\[ \|\mathscr{F}^{-1} (g(\xi)\,e^{-tb(\xi)})\|_{L^r}\lesssim t^{-\frac{n}\theta\left(1-\frac1r\right)-\frac\eta\theta}, \]
for any $r\in[1,\infty]$, if~$\eta>0$ and for any $r\in(1,\infty]$ such that
\begin{equation*}
n\left(1-\frac1r\right)>-\eta,
\end{equation*}
if $\eta\in(-n,0]$.
\end{lemma}

The following lemma is used at low and high frequencies.

\begin{lemma}[see~Corollary 1 in~\cite{DA26}]\label{lem:HF}
    Assume that $m \in  L^\infty$ and verifies
  \[ |\partial_\xi^\alpha m(\xi)|\leq C\,| \xi |^{-|\alpha|+\varepsilon}\,\langle \xi \rangle^{-2\varepsilon}, \qquad \textit{ for } |\alpha|=n-1,n,n+1,\]
  for some $\varepsilon\in (0,1)$. Then $m \in M_1^1$.
\end{lemma}

The following result is a modified version of the Littman lemma about the stationary phase tecnique for oscillatory integrals.

\begin{lemma}\label{lem:Littman}
Let~$v\in \mathcal{C}_c^n(\Rn)$ and~$w\in\mathcal{C}^{n+1,1}(K)$, where $K$ is a compact containing an open set containing $\supp v$. Moreover, we assume that the Hessian matrix~$H_w(\xi)$ satisfies~$\det H_w(\xi)\neq 0$ on~$K$. Let $D$ be such that $\supp v\subset B(0,D)$, and let
\[ M_N = \|v\|_{\mathcal C^N} + \sum_{j=1}^N \left(1+\|H_w\|_{\mathcal C^{j-1,1}(K)}\right)^{j}\|v\|_{\mathcal C^{N-j}}. \]
Then the oscillatory integral
\[
I(\tau) = \int_{\R^n} e^{-i\tau w(\xi)}v(\xi) \, d\xi,
\]
verifies the following estimate:
\begin{equation}\label{eq:New Littman}
|I(\tau)| \leq C_n\,M_n\,D^n\,(1+\tau)^{-\frac n2}\,\log (e+ \tau),
\end{equation}
where the constant~$C_n$ depends on $n$ and on $H_w$ (more precisely, on its eigenvalues), but it is independent of the derivatives of $H_w$ and of $v$. If we strengthen the regularity assumptions to~$v\in \mathcal{C}_c^{n+1}(\Rn)$ and~$w\in\mathcal{C}^{n+2,1}$, then the following estimate holds: 
\begin{equation}\label{eq:Old Littman}
|I(\tau)| \leq C_n\,M_{n+1}\,D^n\,(1+\tau)^{-\frac n2}.
\end{equation}
\end{lemma}
%
The proof of Lemma~\ref{lem:Littman} is given in~\textsection\ref{sec:Littman}. An immediate consequence of Lemma~\ref{lem:Littman} is the following.
\begin{corollary}\label{cor:Littman}
In the same assumptions of Lemma~\ref{lem:Littman}, the thesis holds if we replace $I(\tau)$ by
\[ \|\mathscr{F}^{-1}_{\xi\to x}(e^{-i\tau w(\xi)}v(\xi))\|_\infty. \]
\end{corollary}
\begin{proof}
Since $v$ is compactly supported, we may apply the Fourier inversion formula and obtain
\[ \mathscr{F}^{-1}_{\xi\to x}(e^{-i\tau w(\xi)}v(\xi)) = (2\pi)^{-n} I(\tau),\quad \text{with phase function}\ w_1(\xi) = w(\xi) - \xi\cdot \frac{x}\tau. \]
Due to $H_{w_1}=H_w$, we may apply Lemma~\ref{lem:Littman}, with the same constants, independent of the parameter $x$. %

This concludes the proof.
\end{proof}

\subsection{Proof of Littman-type Lemma~\ref{lem:Littman}}\label{sec:Littman}

We prove Lemma~\ref{lem:Littman}.
\begin{proof}
We follow the ideas contained in \cite[Proposition 2.5, Chapter 8]{SS}, but with some modifications. Let $\varepsilon>0$, sufficiently small to get
\begin{align}
\label{eq:suppeps}
& \supp v + \bar B(0,\varepsilon) = \{x+y: \ x\in\supp v, \ |y|\leq \varepsilon\} \subset K, \\
\label{eq:deteps}
& \varepsilon \|\nabla H_w\|_\infty \leq h_0/4,\qquad \text{where}\quad h_0=\min_{|u|=1} |H_w\cdot u|>0.
\end{align}
It is enough to show estimates \eqref{eq:New Littman} and~\eqref{eq:Old Littman} for $I(\tau)$, assuming that the support of $v$ lies in $B(x_0,\varepsilon)$ for some $x_0$ varying in the compact support of $v$, and then we will obtain the estimate for a general $v$ using a partition of unity. We want to estimate
\[\begin{split}
|I(\tau)|^2= \overline{I(\tau)} I(\tau)
    & = \int_{B(x_0,\varepsilon)} \int_{B(x_0,\varepsilon)} e^{i \tau (w(y)-w(x))}v(y)\,\overline{v}(x)\,dydx\\
    & = \int_{B(x_0,\varepsilon)} \int_{B(x_0-x,\varepsilon)} e^{i \tau (w(x+u)-w(x))}v(x+u)\,\overline{v}(x)\,dudx\\
    & = \int_{B(0,2\varepsilon)} J_\tau(u)\,du,\\
J_\tau(u)
    & = \int_{B(x_0-u,\varepsilon)} e^{i \tau (w(x+u)-w(x))}\varphi(x,u)\,dx,
\end{split}\]
where we used the change of variable $u=y-x$, we put $\varphi(x,u)=v(x+u)\,\overline{v}(x)$, and we applied Fubini-Tonelli theorem. We notice that $B(x_0-u,\varepsilon)\subset B(x_0,3\varepsilon)$. We claim that 
\begin{equation}\label{eq:claim_newLit}
| J_\tau(u)| \leq C_n\,(3\varepsilon)^n\,M_N\,(\tau |u|)^{-N}, \quad N=0,\ldots,n.
\end{equation}
We notice that
\[e^{i \tau (w(x+u)-w(x))} = L(e^{i \tau (w(x+u)-w(x))}),\]
where
\[
L= \frac1{i \tau} (a \cdot \nabla_x), \quad a=\frac{b}{|b|^2},\quad b(x,u)=\nabla_x(w(x+u)-w(x)).
\]
Due to $|u|\leq 2\varepsilon$, for any $x\in\supp v$ and $x+u\in\supp v$, we deduce that $x+tu\in K$ for any $t\in[0,1]$ thanks to~\eqref{eq:suppeps}; hence, we may apply Taylor's theorem to get
\[ |\nabla_x(w(x+u)-w(x))-H_w(x) \cdot u| \leq \|\nabla H_w\|_\infty\,|u|^2;\]
thanks to~\eqref{eq:deteps}, using $|H_w(x) \cdot u|\geq h_0\,|u|$, we obtain
\[ |b(x,u)| \geq \big(h_0-2\varepsilon \|\nabla H_w\|_\infty\big)\,|u| \geq \frac{h_0}2\,|u|, \quad \text{so that}\qquad \frac1{|b|}\leq \frac2{h_0|u|}. \]
We immediately deduce
\[|a(x,u)|\leq\frac2{h_0|u|}.\]
On the other hand, due to $w\in\mathcal C^{n+1,1}$, we get the estimate
\[ |\partial_x^\alpha b(x,u) | \leq \|\partial_x^\alpha b\|_{\mathcal C^{0,1}}\,|u| \leq \|H_w\|_{\mathcal C^{|\alpha|-1,1}(K)}\, |u|, \quad |\alpha|=1,\ldots,n. \]
%
%
%
Therefore,
\[ |\partial_x^\alpha a(x,u) |\leq C_n\,\frac1{|u|}\left(1 + \|H_w\|_{\mathcal C^{|\alpha|-1,1}(K)}\right)^{|\alpha|}, \qquad |\alpha|=1,\ldots,n,\]
where $C_n$ depends on $h_0$. Since
\[ L^t(f)=\frac i \tau \nabla \cdot (af) = \frac{i}\tau\,((\nabla\cdot a)f+a\cdot \nabla f),\]
by the recursive formula
\[ (L^t)^N(f)= \frac{i}\tau\,((\nabla\cdot a)(L^t)^{N-1}(f)+a\cdot \nabla (L^t)^{N-1}(f)), \]
we may estimate
\[ |(L^t)^N \varphi(x,u)| \leq C_n\,\frac1{(\tau |u|)^{N}}\,\Big(\|v\|_{\mathcal C^N} + \sum_{j=1}^N \left(1+\|H_w\|_{\mathcal C^{j-1,1}(K)}\right)^{j}\|v\|_{\mathcal C^{N-j}}\Big),\]
for $N=0,\ldots,n$. Therefore, from $x\in B(x_0,3\varepsilon)$, we get \eqref{eq:claim_newLit}. We may now estimate $|I(\tau)|$. Letting $\tau\geq\varepsilon^{-1}$, we use~\eqref{eq:claim_newLit} with $N=0$ for $u\in B(0,\tau^{-1})$ and~\eqref{eq:claim_newLit} with $N=n$ for $\tau\leq |u|\leq 2\varepsilon$. Then
\[\begin{split}
|I(\tau)|^2
    & \leq C_0\,(3\varepsilon)^n\,\|v\|_\infty\,\int_{|u|\leq \tau^{-1}}du \\
    & \qquad + C_n\,(3\varepsilon)^n\,M_n\,\int_{\tau^{-1} \leq |u|\leq 2\varepsilon} (\tau |u|)^{-n}\,du \\
    & \leq C_1\,(3\varepsilon)^n\,\left(\|v\|_\infty+M_n\right)\, \big( \tau^{-n}+ \tau^{-n} \log ( 2\varepsilon \tau) \big).
\end{split}\] 
Due to $\tau\geq\varepsilon^{-1}$, for $\varepsilon\leq1/2$, we get
\[ 0<\log ( 2\varepsilon \tau) \leq \log \tau. \]
The quantity of balls of radius $\varepsilon$ needed to cover the support of $v$ can be estimated by the measure of a ball containing $\supp v$, multiplied by $\varepsilon^{-n}$; hence, by $c_nD^n\varepsilon^{-n}$. This concludes the proof of \eqref{eq:New Littman}. In order to prove~\eqref{eq:Old Littman}, it is sufficient to use $N\leq n+1$ in all previous steps and~\eqref{eq:claim_newLit} with $N=n+1$ for $\tau\leq |u|\leq 2\varepsilon$, obtaining
\[ |I(\tau)|^2 \leq C\,(3\varepsilon)^n\,M_{w,n+2}\,K_{n+1}\,\tau^{-n}.\]

This concludes the proof
\end{proof}

\section*{Acknowledgements}
Antonio Lagioia is supported by the European Union ERC CONSOLIDATOR GRANT 2023 GUnDHam, Project Number: 101124921. Views
and opinions expressed are however those of the author only and do not necessarily reflect those of the European Union or the
European Research Council. Neither the European Union nor the granting authority can be held responsible for them. Marcello D'Abbicco is supported by PRIN 2022 ``Anomalies in partial differential equations and applications'' CUP H53C24000820006.
Both authors are members of INdAM-GNAMPA group.

\section*{Declarations}

\begin{itemize}
\item Author contribution: All authors have contributed equally in the completion of
the manuscript.
\item Data availability:Not applicable (this manuscript does not report data generation
or analysis).
\item Conflict of interest/Competing:  interests: The authors declare no conflicts of interest.
\item Ethics approval and consent to participate: The authors approved ethics and
consented to participate.
\item Consent for publication: The authors consented to this publication.

\item Materials availability: Not applicable (this manuscript does not report materials
generation or analysis).

\end{itemize}

\end{document}